\documentclass{birkjour}
\usepackage{amssymb,amsthm,amsmath}
\newtheorem{thm}{Theorem}[section]
\newtheorem{theorem}[thm]{Theorem}
\newtheorem{corollary}[thm]{Corollary}
\newtheorem{lemma}[thm]{Lemma}

\theoremstyle{definition}

\theoremstyle{remark}
\newtheorem{remark}[thm]{Remark}
\newtheorem*{ex}{Example}

\numberwithin{equation}{section}

\newcommand{\BC}{\mathcal{B}}

\newcommand{\dr}{\mathrm{d}}
\newcommand{\de}{\mathrm{\delta}}
\newcommand{\DR}{\mathrm{D}}

\newcommand{\eps}{\varepsilon}

\newcommand{\KC}{\mathcal{K}}
\newcommand{\vk}{\varkappa}
\newcommand{\la}{\lambda}
\newcommand{\La}{\Lambda}
\newcommand{\LC}{\mathcal{L}}
\newcommand{\NC}{\mathcal{N}}
\newcommand{\NR}{\mathrm{N}}

\newcommand{\Om}{\Omega}

\newcommand{\Si}{\Sigma}

\newcommand{\R}{\mathbb{R}}
\newcommand{\Sbb}{\mathbb{S}}

\newcommand{\const}{\mathrm{const}}
\newcommand{\diag}{\mathrm{diag}}
\newcommand{\dist}{\mathrm{dist}}

\newcommand{\off}{\mathrm{off}}

\newcommand{\Tr}{\mathrm{Tr\,}}

\def\leq {\leqslant}
\def\geq {\geqslant}

\begin{document}

\title[Counting Function of Integral Operators]
{Lower Bounds for the Counting Function\\ of Integral Operators}

\author{Y. Safarov}
\address{Department of Mathematics\\ 
King's College London\\
Strand, London WC2R 2LS\\
United Kingdom} 
\email{yuri.safarov@kcl.ac.uk}

\dedicatory{To Vladimir Rabinovich  on his 70th birthday}

\begin{abstract}
The paper presents a lower bound for the number of eigenvalues of an integral operator $\,K\,$ with continuous kernel $\,\KC\,$ lying in the interval $\,(-\infty,t)\,$ with $\,t\leq0\,$. The estimate is given in terms of some integrals of $\,\KC\,$. 

\end{abstract}

\date{March 2012}

\thanks{The research was partly carried out during my visit to Academia Sinica, Taipei. I am very grateful to my hosts, especially to Dr. Jin-Cheng Jiang, for their financial and scientific support.
}



\maketitle

\section*{Introduction}

Consider a self-adjoint integral operator $\,K\,$ in the space $L_2(M,\nu)$ on a domain or a manifold $\,M\,$ provided with a finite measure $\,\nu\,$. If its integral kernel $\,\KC\,$ is continuous then the operator is compact and its spectrum consists of eigenvalues accumulating to zero. Such operators have been considered by many authors, most of whom studied the rate of convergence of eigenvalues and obtained various quantitative versions of the following general statement: {\it the smoother the kernel is, the faster the eigenvalues tend to zero} (see, for instance, \cite{BS} or \cite{K}).

The paper deals with a different, seemingly simple question: {\it how many negative eigenvalues are there?} More precisely, we are interested in obtaining explicit lower bounds for the number of negative eigenvalues in terms of the integral kernel $\,\KC\,$.

One can argue that in the generic case the dimensions of positive and negative eigenspaces must be the same, so that both of them are infinite dimensional and there are infinitely many negative eigenvalues. However, this argument is of little use when we need to study a particular integral operator.

It is not immediately clear what properties of $\,\KC\,$ guarantee that there are many negative eigenvalues. The fact that $\,\KC\,$ is real and negative on a large set is clearly insufficient (for instance, the operator with constant integral kernel $\,\KC\equiv-1\,$ has only one negative eigenvalue). On the positive side, if $\,\KC\,$ takes large negative values on the diagonal and the Hilbert--Schmidt norm of $\,K\,$ is relatively small, one can estimate the number of its negative eigenvalues as follows.

\begin{ex}
Let $\,M_-=\{\xi\in M\,:\,\KC(\xi,\xi)<0\}\,$. If $\,M_-\ne\varnothing\,$ then the number of negative eigenvalues of the operator $\,K\,$ is not smaller than 
$$
C_-\ :=\ \left(\int_{M_-}\KC(\xi,\xi)\,\dr\nu(\xi)\right)^2\left(\int_{M_-}\int_{M_-}|\KC(\xi,\eta)|^2\,\dr\nu(\xi)\,\dr\nu(\eta)\right)^{-1}.
$$
Indeed, if $\,K_-\,$ is the truncation of $\,K\,$ to the subspace $\,L_2(M_-,\nu)\,$ then $\,C_-=(\Tr K_-)^2\,\|K_-\|_2^{-2}\,$ where $\,\Tr\,$ and $\,\|\cdot\|_2\,$ stand for the trace and the Hilbert--Schmidt norm. Since $\,\Tr K_-<0\,$, we have
$$
(\Tr K_-)^2\,\|K_-\|_2^{-2}\ \leq\ \left(\sum_j\la_j\right)^2\left(\sum_j\la_j^2\right)^{-2}\ \leq\ \#\{\la_j\}\,,
$$
where $\la_j\,$ are the negative eigenvalues of $\,K_-\,$. Thus $\,C_-\,$ estimates the number of negative eigenvalues of $\,K_-\,$ and, consequently, of $\,K\,$ from below.
\end{ex}

The main result of the paper is Theorem \ref{T:integral} which provides a similar estimate involving some integrals of $\,\KC\,$. Unlike in the previous example, it
does not rely only on the behaviour of $\,\KC\,$ on the diagonal and takes into account the contribution of its off-diagonal part. 

Theorem \ref{T:integral} is stated and proved in Section \ref{S:main}. It is formulated in a very general setting but even in the simplest situation (say, for integral operators on a line segment) the result is not obvious. Section \ref{S:examples} contains some comments and examples. In particular, in Subsection \ref{S:DN} we discuss the link between the problems of estimating the number of negative eigenvalues and the difference between the Dirichlet and Neumann counting functions of the Laplace operator on a domain.

\section{The main theorem}\label{S:main}

Throughout the paper $\,\NC(A;t)\,$ denotes  the dimension of the eigenspace of
a self-adjoint operator (or a Hermitian matrix) $\,A\,$
corresponding to the interval $\,(-\infty,t)\,$.

Let $\,M\,$ be a Hausdorff topological space equipped with a locally
finite Borel measure $\,\nu\,$. We shall always be assuming that $\,M\,$ and  $\,\nu\,$ satisfy the following condition,
\medskip
\begin{enumerate}
\item[{\bf(C$_1$)}]
\ every open set $\,U\subset M\,$ contains infinitely many elements and has non-zero measure.
\end{enumerate}
\medskip

Let us consider the symmetric integral operator $K_0$ in the space $L_2(M,\nu)$ given by a continuous kernel $\,\KC(\eta,\xi)=\overline{\KC(\xi,\eta)}\,$,
\begin{equation}\label{K1}
K_0\ :\ u(\eta)\ \mapsto\ K_0u(\xi)\ :=\
\int_M\KC(\xi,\eta)\,u(\eta)\,\dr\nu(\eta)\,.
\end{equation}
We assume that the domain of $\,K_0\,$ consists of $\,L_2$-functions $\,u\,$ such that the integral on the right hand side of \eqref{K1} is absolutely convergent for almost all $\xi\in M\,$ and the function $\,K_0u\,$, defined by this integral, belongs to the space $\,L_2(M,\nu)\,$. Let $\,K\,$ be an arbitrary self-adjoint extension of $\,K_0\,$.

Let $\,\vk(\xi,\eta)\,$ be the smaller eigenvalue of the
Hermitian $\,2\times2$-matrix
\begin{equation}\label{KC2}
\KC^{(2)}(\xi,\eta)\
:=\ \begin{pmatrix}\KC(\xi,\xi) & \KC(\xi,\eta)\\
\KC(\eta,\xi) &\KC(\eta,\eta)\end{pmatrix}\,,
\end{equation}
that is, 
\begin{multline}\label{kappa}
\vk(\xi,\eta)\ =\ \frac{\KC(\xi,\xi)+\KC(\eta,\eta)}2\\
-\;\frac12\,\sqrt{\left(\KC(\xi,\xi)-\KC(\eta,\eta)\right)^2
+4\left|\KC(\xi,\eta)\right|^2}\,.
\end{multline}
Obviously, $\,\vk(\xi,\eta)\,$ is a continuous real-valued function on $\,M\times M\,$ such that $\,\vk(\xi,\eta)=\vk(\eta,\xi)\,$.

\begin{remark}\label{R:kappa}
By \eqref{kappa}, if $\,\KC(\xi,\xi)\,$ is identically equal to a constant $\,C\,$ then $\,\vk(\xi,\eta)=C-|\KC(\xi,\eta)|\,$.
\end{remark}

We shall say that a measure $\,\mu\,$ on $\,M\times M\,$ is {\sl
symmetric\/} if it is invariant with respect to the transformation
$\,(\xi,\eta)\mapsto(\eta,\xi)\,$. If $\,\mu\,$ is a symmetric measure on $\,M\times M\,$, we shall denote by $\,\mu'\,$ its marginal, that is, the measure on
$\,M\,$ such that $\,\mu'(S)=\mu\left(S\times M\right)$
for all measurable $\,S\subset M\,$. 
Finally, assuming that
\medskip
\begin{enumerate}
\item[{\bf(C$_2$)}]
\ $\,0<\int_{M\times M}\left(t-\vk(\xi,\eta)\right)_+\dr\mu(\xi,\eta)<\infty\,$
\end{enumerate}
\medskip
where 
$$
\left(t-\vk(\xi,\eta)\right)_+\ :=\ \begin{cases}t-\vk(\xi,\eta)&\text{if\ }t-\vk(\xi,\eta)\geq0\,,\\ 0&\text{if\ }t-\vk(\xi,\eta)<0\,,\end{cases}
$$
let us denote
\begin{equation}\label{constant}
C_t(\mu)\ :=\
\frac{\left(\int_{M\times M}\left(t-\vk(\xi,\eta)\right)_+\dr\mu(\xi,\eta)\right)^2}{
\int_M\int_M|\KC(\xi,\eta)|^2\,\dr\mu'(\xi)\,\dr\mu'(\eta)}\,.
\end{equation}

\begin{theorem}\label{T:integral}
Let the condition {\bf(C$_1$)} be fulfilled. If 
$\;\inf\vk<t\leq0\,$  then
\begin{equation}\label{estimate}
\NC(K,t)\ \geq\ \frac12\;+\;\frac{C_t(\mu)}{16}
\end{equation}
for all symmetric Borel measures $\,\mu\,$ satisfying the condition  {\bf(C$_2$)}.
\end{theorem}

\begin{proof}  Consider the open set
$$
\Si_t\ :=\ \{(\xi,\eta)\in M\times
M\,:\,\vk(\xi,\eta)<t\}\,,
$$
and let $\,M_t\,$ be its projection onto $\,M\,$,
$$
M_t\ :=\ \left\{\xi\in M\,:\,(\xi,M)\bigcap
\Si_t\ne\varnothing\right\}.
$$
Further on, without loss of generality, we shall be assuming that $\,\mu\,$ is supported on $\,\Si_t\,$ (if not, we replace $\,\mu\,$ with its restriction to $\,\Si_t\,$).

Given a collection
$$
\theta_n\ =\ \left((\xi_1,\eta_1),\dots,(\xi_n,\eta_n)\right)\in
\Si_t^n\ :=\ \underbrace{\Si_t\times\dots\times\Si_t}_{n\
\text{times}}
$$
of points $\,(\xi_j,\eta_j)\in\Si_t\,$, let us consider the
Hermitian $\,2n\times2n$-matrix
$$
\KC^{(2n)}(\theta_n)\ :=\ \begin{pmatrix} \KC_{1,1} &
\KC_{1,2}& \cdots &\KC_{1,n}\\
\KC_{2,1}&\KC_{2,2}& \cdots &\KC_{2,n}\\
\vdots & \vdots & \ddots &\vdots\\
\KC_{n,1} & \KC_{n,2} & \cdots & \KC_{n,n}
\end{pmatrix}
$$
where 
$$
\KC_{i,j}\ :=\ \begin{pmatrix}\KC(\xi_i,\xi_j) & \KC(\xi_i,\eta_j)\\
\KC(\eta_i,\xi_j) &\KC(\eta_i,\eta_j)\end{pmatrix}\ =\
\left(\KC_{j,i}\right)^*.
$$
Let 
$$
\tilde\KC^{(2n)}(\theta_n)\ :=\ \La\left(\KC^{(2n)}(\theta_n)-tI\right)\La\,,
$$
where $\,I\,$ is the identity $\,2n\times2n$-matrix
and $\,\La=\diag\{\La_1,\La_2,\ldots,\La_n\}\,$ is the block
diagonal matrix formed by the $\,2\times2$-matrices
$$
\La_j\ =\ \left(t-\vk(\xi_j,\eta_j)\right)^{-1/2}\,\begin{pmatrix}1 & 0\\
0 &1\end{pmatrix}\,.
$$
Since $\,\La>0\,$, we have
\medskip
\begin{enumerate}
\item[{\bf(i)}]
$\,\NC\left(\KC^{(2n)}(\theta_n)\,,\,t\right)
=\NC\left(\tilde\KC^{(2n)}(\theta_n)\,,\,0\right)\,$ for all
$\,\theta_n\in\Si_t^n\,$.
\end{enumerate}
\medskip

By direct calculation,
$$
\tilde\KC^{(2n)}(\theta_n)\ :=\ \begin{pmatrix}
\tilde\KC_{1,1}-t\La_1^2 &
\tilde\KC_{1,2}& \cdots &\tilde\KC_{1,n}\\
\tilde\KC_{2,1}&\tilde\KC_{2,2}-t\La_2^2& \cdots &\tilde\KC_{2,n}\\
\vdots & \vdots & \ddots &\vdots\\
\tilde\KC_{n,1} & \tilde\KC_{n,2} & \cdots &
\tilde\KC_{n,n}-t\La_n^2
\end{pmatrix}
\,,
$$
where 
$$
\tilde\KC_{i,j}\ :=\ \left(t-\vk(\xi_i,\eta_i)\right)^{-1/2}
\left(t-\vk(\xi_j,\eta_j)\right)^{-1/2}\,\KC_{i,j}\ =\ (\tilde\KC_{j,i})^*\,.
$$
Let us split $\,\tilde\KC^{(2n)}(\theta_n)\,$ into the sum of the
block diagonal matrix
$$
\tilde\KC^{(2n)}_\diag(\theta_n)\ :=\
\diag\left\{\tilde\KC_{1,1}-t\La_1^2\,,\,\tilde\KC_{2,2}-t\La_2^2\,,\,\ldots\,,\,
\tilde\KC_{n,n}-t\La_n^2\right\}
$$
and the matrix $\,\tilde\KC_\off^{(2n)}(\theta_n):=
\tilde\KC^{(2n)}(\theta_n)-\tilde\KC^{(2n)}_\diag(\theta_n)\,$. The equalities
$$
\tilde\KC_{j,j}-t\La_j^2\ =\
\left(t-\vk(\xi_j,\eta_j)\right)^{-1}\left(\KC^{(2)}(\xi_j,\eta_j)-tI\right),\qquad j=1,\ldots,n\,,
$$
imply that
\medskip
\begin{enumerate}
\item[{\bf(ii)}]
$\,-1\,$ is an eigenvalue of $\,\tilde\KC^{(2n)}_\diag(\theta_n)\,$
of multiplicity $\,n\,$ or higher for each $\,\theta_n\in\Si_t^n\,$.
\end{enumerate}
\medskip

On the other hand,
\begin{multline}\label{off-diag}
\|\tilde\KC_\off^{(2n)}(\theta_n)\|_2^2\\
=\ \sum_{i\ne j}
\frac{|\KC(\xi_i,\xi_j)|^2+|\KC(\xi_i,\eta_j)|^2
+|\KC(\eta_i,\xi_j)|^2+|\KC(\eta_i,\eta_j)|^2}
{\left(t-\vk(\xi_i,\eta_i)\right)\,\left(t-\vk(\xi_j,\eta_j)\right)}\;,
\end{multline}
where $\,\|\cdot\|_2\,$ is the Hilbert--Schmidt norm.
Let us consider the absolutely continuous with respect to $\,\mu\,$
measure $\,\tilde\mu\,$ with the density
$\,\left(t-\vk(\xi,\eta)\right)\,$, so that
$\,\dr\tilde\mu(\xi,\eta)=\left(t-\vk(\xi,\eta)\right)\dr\mu(\xi,\eta)\,$.
Then
\begin{multline*}
\int_{\Si_t^n}\left(t-\vk(\xi_i,\eta_i)\right)^{-1}\left(t-\vk(\xi_j,\eta_j)\right)^{-1}\,
|\KC(\xi_i,\xi_j)|^2\,\dr\tilde\mu(\xi_1,\eta_1)\ldots\dr\tilde\mu(\xi_n,\eta_n)\\
=\left(\tilde\mu(\Si_t)\right)^{n-2}\int_{\Si_t}\int_{\Si_t}
|\KC(\xi_i,\xi_j)|^2\,\dr\mu(\xi_i,\eta_i)\,\dr\mu(\xi_j,\eta_j)
=(C_t(\mu))^{-1}\left(\tilde\mu(\Si_t)\right)^n
\end{multline*}
for all $\,i\ne j\,$, where $\,C_t(\mu)\,$ is defined by
\eqref{constant} and $\,\tilde\mu(\Si_t)\,$ is finite in view of {\bf(C$_2$)}. Similar calculations show that the integrals over
$\,\Si_t^n\,$ with respect to
$\,\dr\tilde\mu(\xi_1,\eta_1)\ldots\dr\tilde\mu(\xi_n,\eta_n)\,$ of
all other term in the right hand side of \eqref{off-diag} are also
equal to $\,(C_t(\mu))^{-1}\left(\tilde\mu(\Si_t)\right)^n\,$.
Therefore
$$
\left(\tilde\mu(\Si_t)\right)^{-n}\int_{\Si_t^n}\|\tilde\KC_\off^{(2n)}(\theta_n)\|_2^2\;
\dr\tilde\mu(\xi_1,\eta_1)\ldots\dr\tilde\mu(\xi_n,\eta_n)\ =\
4\,n(n-1)\,(C_t(\mu))^{-1}
$$
and, consequently, there exists a point
$\,\theta_{n,0}\in\Si_t^n\,$ such that
$$
\|\tilde\KC_\off^{(2n)}(\theta_{n,0})\|_2^2\
\leq\ 4\,n(n-1)\,(C_t(\mu))^{-1}.
$$

Since $\,\|\tilde\KC_\off^{(2n)}(\theta_{n})\|_2^2\,$ continuously
depends on $\,\theta_n\,$ and every open set contains infinitely many elements, for each $\,\eps>0\,$ there exists a point
$$
\theta_{n,\eps}\ =\
\left((\xi_{1,\eps},\eta_{1,\eps}),\dots,(\xi_{n,\eps},\eta_{n,\eps})\right)\
\in\ \Si_t^n\,
$$
such that 
\begin{equation}\label{HS}
\|\tilde\KC_\off^{(2n)}(\theta_{n,\eps})\|_2^2\ \leq\
4\,n(n-1)\,(C_t(\mu))^{-1}+\eps
\end{equation}
and all the entries $\,\xi_{i,\eps}\,$ and $\,\eta_{j,\eps}\,$
are distinct. The estimate \eqref{HS} implies that the number of eigenvalues of the
matrix $\,\tilde\KC_\off^{(2n)}(\theta_{n,\eps})\,$ lying in the
interval $\,[1,\infty)\,$ does not exceed
$\,4\,n(n-1)\,(C_t(\mu))^{-1}+\eps\,$. Therefore, in view of
{\bf(i)} and {\bf(ii)},
$$
\NC\left(\KC^{(2n)}(\theta_{n,\eps})\,,\,t\right)=
\NC\left(\tilde\KC^{(2n)}(\theta_{n,\eps})\,,\,0\right)\geq
n-4\,n(n-1)\,(C_t(\mu))^{-1}-\eps\,.
$$

Since the measure $\,\nu\,$ is locally finite and the function
$\,\KC\,$ is continuous, for every $\,\de>0\,$ and 
$\,(\xi,\eta)\in M\times M\,$ there exist open neighbourhoods
$\,U_{\xi,\de}\subset M\,$ and $\,U_{\eta,\de}\subset M\,$ of the
points $\,\xi\,$ and $\,\eta\,$ such that
$\,\nu(U_{\xi,\de})<\infty\,$, $\,\nu(U_{\eta,\de})<\infty\,$ and
$$
|\KC(\xi,\eta)-\KC(\xi',\eta')|\ <\ \de\,,\qquad\forall\xi'\in
U_{\xi,\de}\,,\quad\forall\eta'\in U_{\eta,\de}\,.
$$
In view of {\bf(C$_1$)}, $\,\nu(U_{\xi,\de})>0\,$ and $\,\nu(U_{\eta,\de})>0\,$. Let $\,u_{\xi,\de}:=\left(\nu(U_{\xi,\de})\right)^{-1}\chi_{\xi,\de}\,$
and
$\,u_{\eta,\de}:=\left(\nu(U_{\eta,\de})\right)^{-1}\chi_{\eta,\de}\,$, where $\,\chi_{\xi,\de}\,$ and $\,\chi_{\eta,\de}\,$ are the
characteristic functions of the sets $\,U_{\xi,\de}\,$ and
$\,U_{\eta,\de}\,$.
Then $\,u_{\xi,\de}\,,\,u_{\eta,\de}\in L_2(M,\nu)\,$ and
\begin{equation}\label{delta}
\left|\KC(\xi,\eta)
-\left(Ku_{\xi,\de},u_{\eta,\de}\right)_{L_2(M,\nu)}\right| \ <\
\de\,.
\end{equation}

Let us choose the neighbourhoods $\,U_{\xi,\de}\,$ and
$\,U_{\eta,\de}\,$ so small that all the functions
$\,u_{\xi_{i,\eps},\de}\,$ and $\,u_{\eta_{j,\eps},\de}\,$ have
disjoint supports, and let $\,K_{\eps,\de}\,$ be the contraction of
$\,K\,$ to the $\,2n$-dimensional subspace spanned by these
functions. In view of \eqref{delta}, in the basis
$$
\{u_{\xi_{1,\eps},\de},\,\ldots,\,u_{\xi_{n,\eps},\de},\,u_{\eta_{1,\eps},\de},\,
\ldots,\,u_{\eta_{n,\eps},\de}\}
$$
the operator $\,K_{\eps,\de}\,$ is represented by a
$\,2n\times2n$-matrix that converges to
$\,\KC^{(2n)}(\theta_{n,\eps})\,$ as $\,\de\to0\,$. Consequently,
$$
\NC(K_{\eps,\de},t)\ \geq\ n-4\,n(n-1)\,(C_t(\mu))^{-1}-\eps
$$ 
for
all sufficiently small $\,\de\,$. By the variational principle, the
same estimate holds $\,\NC(K,t)\,$. Letting $\,\eps\to0\,$, we see
that
\begin{multline*}
\NC(K,t)\ \geq\ n-4\,n(n-1)\,(C_t(\mu))^{-1}\\
=\
\frac1{C_t(\mu)}\,\left(\left(\frac{C_t(\mu)+4}{4}\right)^2
-\left(2n-\frac{C_t(\mu)+4}{4}\right)^2\right)
\end{multline*}
for all positive integers $\,n\,$.
Choosing $\,n_\mu\geq1\,$ such that
$\,\left|2n_\mu-\frac{C_t(\mu)+4}{4}\right|\leq1\,$ and
substituting $\,n=n_\mu\,$ in the above inequality, we obtain
\eqref{estimate}.
\end{proof}

\section{Comments and examples}\label{S:examples}

\subsection{General comments}\label{S:remarks}

The estimate \eqref{estimate} implies that $\,K\,$ has at least one eigenvalue below a negative $\,t\,$ whenever $\,\inf\vk<t\,$. Indeed, in this case \eqref{estimate} with any symmetric measure $\mu$ satisfying {\bf(C$_2$)} shows that $\,\NC(K,t)\geq\frac12\,$. Since the function $\,\NC(K,t)\,$ is integer-valued, it follows that $\,\NC(K,t)\geq1\,$.

If $\,C_t(\mu)\leq8\,$ then \eqref{estimate} implies only
the obvious estimate $\,\NC(K,t)\geq1\,$. In order to obtain a
better result, one has to increase the constant $\,C_t(\mu)\,$ by
choosing an appropriate measure $\,\mu\,$. In particular, Theorem \ref{T:integral} gives a good estimate when the function $\,\vk(\xi,\eta)\,$ is takes large negative values on a ``thin'' subset $\,\Si'\subset M\times
M\,$, the measure $\,\mu\,$ is supported on $\,\Si'\,$ and
$\,|\KC|\,$ is relatively small outside a neighbourhood of $\,\Si'\,$. On the
contrary, if $\,\KC\,$ is almost constant on $\,M\times M\,$ then $\,\vk\approx\KC-|\KC|\,$ and
$\,C_t(\mu)\approx(t-\KC+|\KC|)_+^2\,|\KC|^{-2}\leq4\,$.

A possible strategy of optimizing the choice of $\mu$ is to fix the marginal $\mu'$ and to maximize $\int\left(t-\vk(\xi,\eta)\right)_+\dr\mu(\xi,\eta)$ over the set of symmetric measures $\mu$ with the fixed marginal. The minimization (or maximization) of an integral of the form $\int f(\xi,\eta)\,\dr\mu(\xi,\eta)$ over the set of measures with fixed marginals is known as Kantorovich's problem. It has been solved for some special functions $f(\xi,\eta)$ (see, for instance, \cite{GM} and references therein). 

I won't elaborate further on this problem, as it requires different techniques (and a different author). Instead, in the rest of the paper we shall consider a couple of examples demonstrating possible applications of Theorem \ref{T:integral}.

\subsection{Operators with difference kernels in $\R^n$}\label{S:difference}
Let $\nu$ be a Borel measure on $\R^n$, and let $h$ be a continuous function on $\R^n$ such that $h(-\theta)=\overline{h(\theta)}$. Consider the symmetric operator  
\begin{equation}\label{diff1}
u(\eta)\ \mapsto\ K_0u(\xi)\ :=\
\int h(\xi-\eta)\,u(\eta)\,\dr\nu(\eta)
\end{equation}
in the space $L_2(\R^n,\nu)$. In the notation of Section \ref{S:main}, $M=\R^n$, $\,\KC(\xi,\eta)=h(\xi-\eta)\,$ and $\,\vk(\xi,\eta)=h(0)-|h(\xi-\eta)|\,$ (see Remark \ref{R:kappa}).

Let us fix $\,t\leq0\,$ and $\theta\in\R^n$ such that $|h(\theta)|>h(0)-t$, and define a measure $\mu_\theta$ on $\,\R^{2n}\,$ by the identity
$$
\int f(\xi,\eta)\,\dr\mu_\theta(\xi,\eta)\ =\ 
\int\left(f(\eta,\eta+\theta)+f(\eta+\theta,\eta)\right)\,\dr\tilde\mu(\eta)\,,
$$
where $\tilde\mu$ is a probability measure on $\R^n$.
Clearly, the measure $\mu_\theta$ is symmetric, and its marginal coincides with the measure $\,\mu'_\theta\,$ on $\,\R^n\,$ given by the equality
$$
\int v(\eta)\,\dr\mu'_\theta(\eta)\ =\ \int \left(v(\eta)+v(\eta+\theta)\right)\,\dr\tilde\mu(\eta)\,.
$$

We have 
\begin{equation}\label{diff2}
\int\left(t-\vk(\xi,\eta)\right)_+\dr\mu_\theta(\xi,\eta)\ =\ 2\left(|h(\theta)|-h(0)+t\right)
\end{equation}
and
\begin{multline}\label{diff3}
\iint|\KC(\xi,\eta)|^2\,\dr\mu'_\theta(\xi)\,\dr\mu'_\theta(\eta)\\ 
=
\iint\left(2|h(\xi-\eta)|^2+|h(\xi-\eta+\theta)|^2+|h(\xi-\eta-\theta)|^2\right)\dr\tilde\mu(\xi)\,\dr\tilde\mu(\eta).
\end{multline}
Since $\,|h(\xi-\eta-\theta)|=|h(\eta-\xi+\theta)|\,$ and
$$
\iint|h(\eta-\xi+\theta)|^2\,\dr\tilde\mu(\xi)\,\dr\tilde\mu(\eta)\ =\ \iint|h(\xi-\eta+\theta)|^2\,\dr\tilde\mu(\xi)\,\dr\tilde\mu(\eta)\,,
$$
the equality \eqref{diff3} can be rewritten in the form
\begin{multline}\label{diff4}
\iint|\KC(\xi,\eta)|^2\,\dr\mu'(\xi)\,\dr\mu'(\eta)\\
=\ 2\iint\left(|h(\xi-\eta)|^2+|h(\xi-\eta+\theta)|^2\right)\,\dr\tilde\mu(\xi)\,\dr\tilde\mu(\eta)\,.
\end{multline}
In view of \eqref{diff2} and \eqref{diff4}, Theorem \ref{T:integral} implies that 
\begin{equation}\label{estimate1}
\NC(K,t)\ \geq\ \ \frac12\;+\; \frac{\left(|h(\theta)|-h(0)+t\right)^2}{8\iint\left(|h(\xi-\eta)|^2+|h(\xi-\eta+\theta)|^2\right)\,\dr\tilde\mu(\xi)\,\dr\tilde\mu(\eta)}
\end{equation}
for all probability measures $\,\tilde\mu\,$ on $\,\R^n\,$.

Let $\,\dr\tilde\mu(\xi)=\eps^n\chi(\eps\xi)\,\dr\xi\,$ where $\,\chi\,$ is the characteristic function of the unit ball. One can easily see that 
\begin{multline*}
\limsup_{\eps\to0}\iint\left(|h(\xi-\eta)|^2+|h(\xi-\eta+\theta)|^2\right)\,\eps^{2n}\chi(\eps\xi)\,\chi(\eps\eta)\,\dr\xi\,\dr\eta\\
\leq\ 2\limsup_{\theta\to\infty}|h(\theta)|^2\,.
\end{multline*}
Passing to the limit in \eqref{estimate1} and optimizing the choice of $\,\theta\,$, we obtain 

\begin{corollary}\label{C:convolution}
Let $\,K\,$ be a self-adjoint extension of the operator \eqref{diff1}.
If the measure $\,\nu\,$ satisfies the condition {\bf(C$_1$)} and $\,h(0)-\sup_{\theta\in\R^n}|h(\theta)|<t\leq0\,$
then
\begin{equation}\label{estimate2}
\NC(K,t)\ \geq\ \frac12\;+\;\left(\frac{\sup_{\theta\in\R^n}|h(\theta)|-h(0)+t}{4\,\limsup_{\theta\to\infty}|h(\theta)|}\right)^2.
\end{equation}
\end{corollary}

In particular, \eqref{estimate2} implies that $\,\NC(K,0)=\infty\,$ whenever $\,h\not\equiv0\,$ and $\,\lim_{\theta\to\infty}|h(\theta)|=0\,$.

\subsection{Dirichlet and Neumann counting
functions}\label{S:DN}
Consider the Laplace operator $\,\Delta\,$ on an open domain $\,\Om\subset\R^d\,$, and denote by $\,N_\DR(\la)\,$ and $\,N_\NR(\la)\,$ the numbers of its Dirichlet and Neumann eigenvalues lying in the interval $\,[0,\la^2)\,$.

Let $\,G_\la:=\left\{f\in L_2(\Om)\,\,:\,-\Delta f=\la^2
f\right\}\,$, where the equality $\,-\Delta u=\la^2f\,$ is
understood in the sense of distributions, and let $\,\BC_\la\,$ be
the self-adjoint operator in $\,G_\la\,$ generated by the truncation
of the quadratic form $\,\|\nabla
f\|_{L_2(\Om)}^2-\lambda^2\,\|f\|_{L_2(\Om)}^2\,$ to the subspace
$\,G_\la\,$. 

\begin{lemma}\label{lem:Fr}
For any open bounded set $\,\Om\,$,
\begin{enumerate}
\item[(1)]
the kernel of $\,\BC_\la\,$ is spanned by the Dirichlet and Neumann
eigenfunctions corresponding to $\,\la^2\,$;
\item[(2)]
$\;N_\NR(\la)-N_\DR(\la)=n_\DR(\la)+g^-(\la)\;$, where
$\,n_\DR(\la)\,$ is the number of linearly independent Dirichlet
eigenfunctions corresponding to the eigenvalue $\,\lambda^2\,$, and
$\,g^-(\la)\,$ is the dimension of the negative eigenspace of
$\,\BC_\la\,$.
\end{enumerate}
\end{lemma}

\begin{proof}
This is a particular case of \cite[Lemma 1.2]{S} and \cite[Theorem
1.7]{S}.
\end{proof}

\begin{remark}
For domains smooth boundaries, Lemma \ref{lem:Fr}(2) was proved in
\cite{Fr}. In this section, we shall only need the estimate
$\;N_\NR(\la)-N_\DR(\la)\leq n_\DR(\la)+g_\la\;$ which can
easily be deduced from the variational principle, using integration by
parts \cite{Fi}. 
\end{remark}

In order to obtain effective estimates with the use of Lemma
\ref{lem:Fr}, we need some information about the space
$\,G_\la\,$. It is not easy to describe, as it depends on $\,\Om\,$.
However, the subspace $\,G_\la\,$ always contains restrictions to
$\,\Om\,$ of the functions $\,f\,$ satisfying the equation
$\,-\Delta f=\la^2f\,$ on the whole space $\,\R^d\,$. In particular, $\,G_\la\,$ contains restrictions to $\,\Om\,$ of the
functions 
\begin{equation}\label{f}
f_u(x)\ =\ \int_{\Sbb_\la^{d-1}}\,e^{-ix\cdot\xi}\,u(\xi)\,\dr\nu(\xi)\,,
\end{equation}
where $\,\nu(\xi)\,$ is a finite Borel measure on the sphere 
$\,\Sbb_\la^{d-1}:=\{\xi\in\R^d:|\xi|=\la\}\,$ and $\,u\,$ is a function from $\,L_2(\Sbb_\la^{d-1},\nu)\,$. Note that the integral in the right hand side of \eqref{f} defines a real analytic function on $\,\R^n\,$, so that
$\,\left.f_u\right|_\Om\not\equiv0\,$ for all nonzero $\,u\in L_2(\Sbb_\la^{d-1},\nu)\,$.

Let $\,K_{\la,\nu}\,$ be the operator in the space $\,L_2(\Sbb_\la^{d-1},\nu)\,$ given by the integral kernel
$$
\KC(\xi,\eta)\ :=\ -\,|\xi-\eta|^2\,\hat\chi_\Om(\xi-\eta)\,,
$$
where $\,\hat\chi_\Om\,$ is the Fourier transform of the characteristic function $\,\chi_\Om\,$ of the set $\,\Om\,$. One can easily see that 
\begin{equation}\label{KC}
\|\nabla f_u\|_{L_2(\Om)}^2-\lambda^2\,\|f_u\|_{L_2(\Om)}^2\ =\ \frac12\,(K_{\la,\nu} u,u)_{L_2(\Sbb_\la^{d-1},\nu)}
\end{equation}
for all $\,u\in L_2(\Sbb_\la^{d-1},\nu)\,$

\begin{corollary}\label{C:DN1}
For all open sets $\,\Om\,$, all $\,\la>0\,$ and all Borel measures $\,\nu\,$ on $\,\Sbb_\la^{d-1}\,$ we have
\begin{equation}\label{DN1}
N_\NR(\la)-N_\DR(\la)\ \geq\ \NC(K_{\la,\nu},0)+n_\DR(\la)\,.
\end{equation}
\end{corollary}

\begin{proof}
Denote by $\,\LC_\la^-\,$ the negative eigensubspace of the operator $\,K_{\la,\nu}\,$,
and let
$\,L_\la^-=\{f_u\,:\,u\in\LC_\la^-\}\,$.
In view of \eqref{KC},    $\,(\BC_\la f,f)_{L_2(\Om)}<0\,$ for all nonzero $\,f\in L_\la^-\,$. By the variational principle, $\,g_-(\la)\geq\dim L_\la^-=\NC(K_{\la,\nu},0)\,$. This inequality and Lemma \ref{lem:Fr}(2) imply \eqref{DN1}.
\end{proof}

One can slightly improve the estimate \eqref{DN1} assuming that
\begin{enumerate}
\item[{\bf(C$_3$)}]
the subspace $\,L_\la\,$ does not contain a Dirchlet or Neumann eigenfunction of the form $\,f_u\,$ with $\,u\in L_2(\Sbb_\la^{d-1},\nu)\,$.
\end{enumerate}

\begin{corollary}\label{C:DN2}
If the condition {\bf(C$_3$)} is fulfilled then
\begin{equation}\label{DN2}
N_\NR(\la)-N_\DR(\la)\ \geq\ \NC(K_{\la,\nu},0)+\dim\ker K_{\la,\nu}+n_\DR(\la)\,.
\end{equation}
\end{corollary}

\begin{proof}
Let $\,L_\la^0=\{f_u\,:\,u\in\ker K_{\la,\nu}\}\,$. By \eqref{KC},  $\,(\BC_\la f,f)_{L_2(\Om)}\leq0\,$ for all  functions $\,f\in L_\la^-+ L_\la^0\,$. Also, Lemma \ref{lem:Fr}(1) and the condition {\bf(C$_3$)} imply that $\,\ker\BC_\la\bigcap\left(L_\la^-+ L_\la^0\right)=\{0\}\,$. Now the standard variational arguments show that
$$
g_-(\la)\geq\dim\left(L_\la^-+ L_\la^0\right)=\NC(K_{\la,\nu},0)+\dim\ker K_{\la,\nu}\,,
$$
and \eqref{DN2} follows from Lemma \ref{lem:Fr}(2).
\end{proof}

\begin{remark}\label{R:DN1}
Since $\,\KC_\la(\xi,\xi)\equiv0\,$, we have $\,\vk(\xi,\eta)=-|\KC_\la(\xi,\eta)|\,$ (see Remark \ref{R:kappa}). Thus $\,\inf\vk<0\,$ and, consequently, $\,\NC(K_{\la,\nu},0)\geq1\,$. Therefore \eqref{DN1} implies the estimate $\,N_\NR(\la)-N_\DR(\la)\geq1+n_\DR(\la)\,$, which was obtained in \cite{Fr} and \cite{Fi}.
\end{remark}

\begin{remark}\label{R:DN2}
If
$\,\hat\chi_\Om(\theta)=0\,$ for some $\,\theta\in\R^d\,$ and $\,\nu\,$ is the sum of $\,\de$-measures at any two points 
$\,\xi,\eta\in\Sbb_\la^{d-1}\,$ such that $\,\xi-\eta=\theta\,$, then $\,K_{\la,\nu}=0\,$. Applying Corollary \ref{C:DN2}, we see that $\,N_\NR(\la)-N_\DR(\la)\geq2+n_\DR(\la)\,$ for all
$\,\la\geq|\theta|/2\,$. This estimate was discusses in \cite{BLP}.
\end{remark}

Since the function $\,\KC\,$ is continuous, it is almost constant for small $\,\xi\,$ and $\,\eta\,$. Therefore Theorem \ref{T:integral} is not well suited for estimating $\,\NC(K_{\la,\nu},0)\,$ with small $\,\la\,$ (see the remark in Subsection \ref{S:remarks}). However, it is useful for studying the behaviour of  $\,N_\NR(\la)-N_\DR(\la)\,$ for large values of $\,\la\,$.

\begin{lemma}\label{L:FS}
Denote 
$$
C_\Om(\la,r)\ =\ \frac{c_{d-1}\,r^4}{18}\left(\inf_{|\theta|=r}|\hat\chi_\Om(\theta)|^2\right)\la^{d-4}\,|\Om_{\la^{-1}}|^{-1}\,,
$$
where $\,c_{d-1}\,$ is the volume of the unit $\,(d-1)$-dimensional sphere in $\,\R^d\,$ and $\,|\Om_{\la^{-1}}|\,$ is the volume of the set $\,\{x\in\Om\,:\,\dist(x,\partial\Om)<\la^{-1}\}\,$. If $\,\nu\,$ is the Euclidean measure on $\,\Sbb_\la^{d-1}\,$ then
$\,\NC(K_{\la,\nu},0)\geq\frac12+\frac{C_\Om(\la,r)}{16}\,$ for all $\,r\in(0,2\la)\,$.

\end{lemma}

\begin{proof}
Let  $\,m_n\,$ be the normalized Euclidean measure on an $\,n$-dimensional sphere $\,\Sbb^n_t:=\{\xi\in\R^{n+1}\,:\,|\xi|=t\}\,$, such that $\,m_n(\Sbb^n_t)=1\,$.
Consider the symmetric probability measure 
$\,\mu_r\,$ on $\,\Sbb_\la^{d-1}\times\Sbb_\la^{d-1}\,$ defined by the equality
\begin{multline*}
\int_{\Sbb_\la^{d-1}\times\Sbb_\la^{d-1}} f(\xi,\eta)\,\dr\mu_r(\xi,\eta)\\
=\  \frac12
\int_{\Sbb_\la^{d-1}}\int_{\eta\in\Sbb_\la^{d-1}:
|\xi-\eta|=r}\left(f(\xi,\eta)+f(\eta,\xi)\right)\dr m_{d-2}(\eta)\,\dr m_{d-1}(\xi)\,.
\end{multline*}

For all functions $\,g\,$ on $\,\Sbb_\la^{d-1}$,  we obviously have 
\begin{multline}\label{marginal1}
\int_{\Sbb_\la^{d-1}}\int_{\eta\in\Sbb_\la^{d-1}:
|\xi-\eta|=r}g(\xi)\,\dr m_{d-2}(\eta)\,\dr m_{d-1}(\xi)
\\
=\ \int_{\Sbb_\la^{d-1}}g(\xi)\,\dr m_{d-1}(\xi)\,.
\end{multline}
On the other hand,
\begin{multline}\label{marginal2}
\int_{\Sbb_\la^{d-1}}\int_{\eta\in\Sbb_\la^{d-1}:|\xi-\eta|\leq r}g(\eta)\,\dr m_{d-1}(\eta)\,\dr m_{d-1}(\xi)\\
=\ \iint_{\Sbb_\la^{d-1}\times\Sbb_\la^{d-1}}\psi_r(\xi,\eta)\,g(\eta)\,\dr m_{d-1}(\eta)\,\dr m_{d-1}(\xi)\\
=\ C_\la(r)\int_{\Sbb_\la^{d-1}}g(\eta)\,\dr m_{d-1}(\eta)\,,
\end{multline}
where $\,\psi_r\,$ is the characteristic function of the set 
$$
\{(\xi,\eta)\in\Sbb_\la^{d-1}\times\Sbb_\la^{d-1}\,:\,|\xi-\eta|\leq r\}
$$
and $\,C_\la(r)=\int_{\xi\in\Sbb_\la^{d-1}:
|\xi-\eta|\leq r}\,\dr m_{d-1}(\xi)\,$. Since 
\begin{multline*}
\frac{\dr}{\dr r}\left(\int_{\eta\in\Sbb_\la^{d-1}:|\xi-\eta|\leq r}g(\eta)\,\dr m_{d-1}(\eta)\right)\\
=\ \frac{c_{d-2}\,r^{d-2}}{c_{d-1}\,\la^{d-1}}
\int_{\eta\in\Sbb_\la^{d-1}:|\xi-\eta|=r}g(\eta)\,\dr m_{d-2}(\eta),
\end{multline*}
differentiating the right and left hand sides of the identity \eqref{marginal2}, we obtain
\begin{multline}\label{marginal3}
\int_{\Sbb_\la^{d-1}}\int_{\eta\in\Sbb_\la^{d-1}:
|\xi-\eta|=r}g(\eta)\,\dr m_{d-2}(\eta)\,\dr m_{d-1}(\xi)\\ 
=\ \int_{\Sbb_\la^{d-1}}g(\eta)\,\dr m_{d-1}(\eta)\,.
\end{multline}
The equalities \eqref{marginal1} and \eqref{marginal3} imply that the marginal $\,\mu'_r\,$ of the measure $\,\mu_r\,$ coincides with $\,m_{d-1}\,$.

Using Remark \ref{R:kappa}, we obtain
\begin{multline*}
\int_{\Sbb_\la^{d-1}\times\Sbb_\la^{d-1}}\left(-\vk(\xi,\eta)\right)_+\,\dr\mu_r(\xi,\eta)=\int_{\Sbb_\la^{d-1}\times\Sbb_\la^{d-1}}|\xi-\eta|^2\,|\hat\chi_\Om(\xi-\eta)|\,\dr\mu_r(\xi,\eta)\\
=
r^2\int_{\Sbb_\la^{d-1}}\int_{\eta\in\Sbb_\la^{d-1}:
|\xi-\eta|=r}
|\hat\chi_\Om(\xi-\eta)|\,\dr m_{d-2}(\eta)\,\dr m_{d-1}(\xi)
\geq r^2\inf_{|\theta|=r}|\hat\chi_\Om(\theta)|.
\end{multline*}
As was shown in \cite{FS},
\begin{multline*}
\int_{\Sbb_\la^{d-1}}\int_{\Sbb_\la^{d-1}}|\KC(\xi,\eta)|^2\,\mu'_r(\xi)\,\mu_r'(\eta)\\
=\ \int_{\Sbb_\la^{d-1}}\int_{\Sbb_\la^{d-1}}|\xi-\eta|^4\,|\hat\chi_\Om(\xi-\eta)|^2\,\dr m_{d-1}(\xi)\,\dr m_{d-1}(\eta)\\ \leq\ 18\,c_{d-1}^{-1}\,\la^{4-d}\,|\Om_{\la^{-1}}|\,.
\end{multline*}
Now the required estimate follows from Theorem \ref{T:integral}.
\end{proof}

Corollary \ref{C:DN1} and Lemma \ref{L:FS} imply that 
\begin{equation}\label{DN3}
N_\NR(\la)-N_\DR(\la)\ \geq\ \const\,\la^{d-4}\,|\Om_{\la^{-1}}|^{-1}
\end{equation}
for all sufficiently large $\,\la\,$. This estimate  was obtained by a different method in \cite{FS}. So far it is unknown whether one can get a better result in terms of growth as $\la\to\infty$ for a general domain $\,\Om\,$. 

For domains with smooth boundaries, the two-term Weyl asymptotic formula (see, for instance, \cite{I} or \cite{SV}) implies that $\,N_\NR(\la)-N_\DR(\la)\geq O(\la^{d-1})\,$. There are reasons to believe that the same is true for all domains but the standard techniques, which work for domains with irregular boundaries, fail to produce such results. It is possible that \eqref{DN3} can be improved by applying Theorem \ref{T:integral} with some other measures $\,\nu\,$ and $\,\mu\,$ to the operator $\,K_{\la,\nu}\,$ or/and more careful analysis of the asymptotic behaviour of the integrals in \eqref{constant} as $\,\la\to\infty\,$.

\end{document}